\documentclass[11pt,letterpaper,reqno]{amsart}
\usepackage{graphicx}
\usepackage{amsmath}
\usepackage{amssymb}
\usepackage{amsfonts}
\usepackage{amsrefs}
\usepackage{amsthm}
\usepackage{cancel}
\usepackage{enumerate}
\usepackage[mathscr]{eucal}

\usepackage{calc}

\usepackage{accents}
\newcommand{\dbtilde}[1]{\accentset{\approx}{#1}}

\newcommand{\R}{\mathbb{R}}

\newcommand{\inv}{^{-1}}
\newcommand{\ol}{\overline}

\newcommand{\sm}{\setminus}

\renewcommand{\div}{\text{div}\,}

\newtheorem{thm}{Theorem}
\newtheorem{lemma}[thm]{Lemma}

\newtheorem{cor}[thm]{Corollary}

\theoremstyle{definition}
\newtheorem{definition}[thm]{Definition}

\theoremstyle{remark}
\newtheorem*{remark}{Remarks}
\newtheorem*{ack}{Acknowledgements}
\newtheorem*{outline}{Outline}
\newtheorem*{motiv}{Motivation}

\DeclareMathOperator{\vol}{vol}
\DeclareMathOperator{\grad}{grad}
\DeclareMathOperator{\Div}{div}

\DeclareMathOperator{\capac}{cap}

\begin{document}
\title{Penrose-type inequalities with a Euclidean background}
\date{\today}
\author{Jeffrey L. Jauregui}
\address{Department of Mathematics\\
Union College\\
807 Union St.\\
Schenectady, NY 12309\\}
\email{jaureguj@union.edu}

\begin{abstract}
The Riemannian Penrose inequality (RPI) bounds from below the ADM mass of asymptotically flat manifolds of 
nonnegative scalar curvature in terms of the total area of all outermost compact minimal surfaces. The general form of the RPI is currently known for manifolds of dimension up to seven.  In the present 
work, we prove a Penrose-like inequality that is valid in all dimensions, for conformally flat manifolds. Our inequality treats the area contributions of
the minimal surfaces in a more favorable way than the RPI, at the expense of using the smaller Euclidean area (rather than the intrinsic area). We give an example
in which our estimate is sharper than the RPI when many minimal surfaces are present. We do not require the minimal surfaces to be outermost.

We also generalize the technique to allow for metrics conformal to a scalar-flat (not necessarily Euclidean) background, and prove a Penrose-type inequality without an assumption on the sign of scalar curvature. Finally, we derive a new lower bound for the ADM mass of a conformally flat, asymptotically 
flat manifold containing any number of zero area singularities.
\end{abstract}

\maketitle

\section{Introduction}
The positive mass theorem (PMT) is a beautiful result on the geometry of manifolds of nonnegative
scalar curvature.  It implies a scalar curvature rigidity statement 
for Euclidean space,  is a crucial ingredient in the solution of the Yamabe problem \cite{schoen_yamabe}, and has deep implications for general relativity.  
The PMT was proved decades ago by Schoen and Yau in dimensions three through seven 
\cites{schoen_yau, schoen_yau2, schoen_lecture_notes} and by Witten for spin manifolds \cite{witten} in dimensions $n \geq 3$. Recently Schoen and Yau have given a proof for all dimensions $n\geq 3$ without the spin assumption \cite{schoen_yau3}.  In section \ref{sec_prelim} we will recall the relevant definitions.
\begin{thm}[Positive mass theorem] 
\label{thm_pmt}
Let $(M,g)$ be a complete, asymptotically flat 
Riemannian $n$-manifold without boundary, with $n \geq 3$.  
Suppose $(M,g)$ has nonnegative scalar curvature and ADM mass $m$.  Then $m \geq 0$, and $m=0$ 
if and only if $(M,g)$ is isometric to $\R^n$ with the Euclidean metric.
\end{thm}
The well-known physical interpretation of the PMT is that in an $(n+1)$-dimensional Lorentzian spacetime obeying the dominant
energy condition, any totally geodesic ``spacelike slice'' (Riemannian submanifold of dimension $n$) has 
nonnegative total mass (see \cites{schoen_yau, bray_RPI} for instance).

A generalization of the PMT is the Riemannian Penrose inequality (RPI), proved as stated below by 
Bray \cite{bray_RPI} (for dimension $n=3$) and later by Bray and Lee \cite{bray_lee} (for $3 \leq n\leq 7$).  
Huisken and Ilmanen gave a proof for $n=3$, with $|\Sigma|_g$ replaced by the area of the largest connected 
component of $\partial M$ \cite{imcf}.  The Bray and Bray--Lee proofs rely on the PMT, while the 
Huisken--Ilmanen approach provides an independent proof of the PMT for $n=3$.
\begin{thm}[Riemannian Penrose inequality]
\label{thm_rpi}
Let $(M^n,g)$ be an asymptotically flat manifold of dimension $3 \leq n \leq 7$ with nonnegative scalar curvature.  
Suppose the boundary $\Sigma = \partial M$ has area $|\Sigma|_g$, zero mean curvature, and is
area-outer-minimizing (see below). Then 
\begin{equation}
\label{ineq_rp}
m_{ADM}(g) \geq \frac{1}{2}\left(\frac{|\Sigma|_g}{\omega_{n-1}}\right)^{\frac{n-2}{n-1}}.
\end{equation}
\end{thm}
We say $\Sigma$ is \emph{area-outer-minimizing} in $M$ if every surface enclosing $\Sigma$ has area at least
as large as $\Sigma$.  The typical physical interpretation of the RPI is as follows (see \cite{bray_RPI} for more details):
view $(M,g)$ as a totally geodesic spacelike slice of a spacetime obeying the dominant energy condition.  
Each component of the minimal boundary $\partial M$ is the apparent horizon of a black hole.  The number 
$\frac{1}{2}\left(\frac{|\Sigma|_g}{\omega_{n-1}}\right)^{\frac{n-2}{n-1}}$ represents the total mass 
of the collection of black holes.  Thus, the RPI states that the total mass of $(M,g)$ is at least the mass contributed by the black holes.

\medskip 
\motiv
In this paper we are largely motivated by a question of Bray and Iga \cite{bray_iga}, regarding whether an elementary proof of the RPI is possible for
conformally flat metrics. In fact, we go a step further and ask if an even sharper inequality is possible, exploiting
the conformally flat structure of $g$. We meet both success and failure: we establish a Penrose-like inequality (Theorem \ref{thm_main}), valid in all dimensions, bounding
the ADM mass from below in terms of the Euclidean area, rather than the $g$-area, of $\Sigma$. Although using the Euclidean area is undesirable (because it is smaller), we
succeed in producing an inequality in which the lower bound accounts for the areas of individual minimal boundary components in a more favorable way. In fact, we will give an example with five or more black holes in which our inequality gives a better estimate for the ADM mass than the RPI. 

\begin{thm}
\label{thm_main}
Let $n\geq 3$, and suppose $M = \R^n \sm \Omega$ for a smooth, mean-convex, bounded open set $\Omega$, 
of which every connected component is star-shaped.  Let $k$ be the number of components of $\Omega$. Suppose $g$ is a Riemannian metric on $M$
of the form $u^{\frac{4}{n-2}} \delta$, where $\delta$ is the Euclidean metric and $u$ is a smooth positive
function approaching one at infinity. Assume that $g$ is asymptotically
flat, has nonnegative scalar curvature $R_g$, and $\partial M = \partial \Omega$ is a minimal surface in $(M,g)$.  Then
\begin{equation}
\label{ineq_main}
m_{ADM}(g) > \left(\frac{A_1}{\omega_{n-1}}\right)^{\frac{n-2}{n-1}} + \ldots
	+ \left(\frac{A_k}{\omega_{n-1}}\right)^{\frac{n-2}{n-1}} + \frac{1}{2(n-1)\omega_{n-1}}
\int_M R_g u\inv dV_g,
\end{equation}
where $A_i$ is the Euclidean area of the boundary of the $i$th connected component of $\Omega$, $\omega_{n-1}$ 
is the Euclidean area of the unit $(n-1)$-sphere in $\R^n$, and $dV_g$ is the volume form of $(M,g)$.
\end{thm}

In addition to the work of Bray and Iga \cite{bray_iga}, who showed  a version of the RPI for dimension three, with suboptimal constant, using only properties of superharmonic functions, we also mention other work on Penrose-like inequalities in special cases. Schwartz and Freire--Schwartz proved ``volumetric''
Penrose inequalities \cites{schwartz, FS} for conformally flat manifolds.  Lam initiated the study of the PMT and RPI for graphs in $\R^{n+1}$ \cites{lam, lam_thesis}, with subsequent work by de Lima--Gir\~ao, \cite{lima_girao}, Huang--Wu \cite{huang_wu}, and others.

We also prove a result similar to Theorem \ref{thm_main}, with \emph{zero area singularities} (ZAS) replacing minimal surfaces. As minimal surfaces model black holes of positive mass, ZAS can be viewed as black holes of negative mass. Bray introduced ZAS \cite{bray_npms}, and further work was done by Robbins \cite{zas_robbins}, the author \cite{thesis}, and Bray and the author \cite{zas}. We refer the reader to Bray's survey paper \cite{bray_survey}. An inequality analogous to the Riemannian Penrose inequality is conjectured for ZAS in an asymptotically flat manifold $(M,g)$ of nonnegative scalar curvature:
\begin{equation}
\label{eqn_z}
m_{ADM}(g) \geq m_{ZAS}(\Sigma),
\end{equation}
where the right-hand side is the ``ZAS mass'' of $\Sigma =\partial M$. (All relevant definitions are recalled in Section \ref{sec_zas}.) For conformally flat manifolds, we are able to prove in Theorem \ref{thm_zas}:
$$m_{ADM}(g) \geq m_{ZAS}(\Sigma) \left(1+\frac{1}{4}\iota^2\right),$$
for a real constant $\iota \geq 1$ depending on the geometry of the underlying Euclidean manifold. This inequality is weaker than (\ref{eqn_z}) because the ZAS mass is negative.

\smallskip
\outline
The paper is organized as follows.  In the next section, we recall some definitions, including asymptotic flatness
and ADM mass.  Section \ref{sec_bh} states and proves some Penrose-like inequalities for metrics
that are conformal to a special background, and constructs a class of examples in which these inequalities are sharper than the RPI.  In Theorem \ref{thm_main}, proved in this section, the background manifold
is Euclidean space minus a union of domains. Theorem \ref{thm_general} is a generalization, allowing for the background to be
any asymptotically flat, zero scalar curvature manifold that agrees to first order with the Euclidean metric on the boundary.  Theorem \ref{thm_general2} is a different generalization that allows for
possibly negative scalar curvature. In section \ref{sec_discuss}, we discuss the implications of these Penrose-like inequalities not requiring the  minimal boundary to be area-outer-minimizing, in contrast with the RPI.

In section \ref{sec_zas}, we recall the details of ZAS.  Theorem \ref{thm_zas} is a type of Penrose inequality for ZAS contained in 
conformally flat manifolds.  Such an inequality is unknown in general (without conformal flatness), even in 
low dimensions. In section \ref{sec_mixed}, we discuss an inequality conjectured by Bray for asymptotically flat manifolds that contain both black holes and zero area singularities.  We apply the techniques of the previous sections to prove a  weaker version of the inequality in the conformally flat case, under an additional technical assumption.

Finally, in an appendix we give a proof of the inequality of Poincar\'e, Faber, and Szeg\"o relating the volume and capacity of regions in $\R^3$ \cite{isoperimetric}; this is needed in the proof of Theorem \ref{thm_zas}.

\ack The author would like to thank Fernando Schwartz and Hubert Bray for helpful 
discussions and suggestions.  He also would like to note that Schwartz's  paper \cite{schwartz} 
was the original inspiration for this work. Much of the work for this paper was carried out while the author was affiliated with the University of Pennsylvania. The author would also like to thank the referee for a number of insightful comments and suggestions that improved the paper.

\section{Preliminaries}
\label{sec_prelim}
\begin{definition}
\label{def_AF}
A smooth, connected, Riemannian manifold $(M,g)$ (possibly with compact boundary) 
of dimension $n\geq 3$ is \textbf{asymptotically flat} (with one end) if 
\begin{enumerate}[(i)]
\item there exists a compact set $K \subset M$ and a diffeomorphism $\Phi: M \sm K \to \R^n \sm \ol{B}$ (where $\ol B$ is a closed ball), and
\item in the ``asymptotically flat coordinates'' $(x^1, \ldots, x^n)$ on $M \sm K$ induced by $\Phi$, the metric $g$ satisfies:
\begin{align*}
 |g_{ij} - \delta_{ij}| &\leq \frac{c}{|x|^p},&
 |\partial_k g_{ij}| &\leq \frac{c}{|x|^{p+1}},\\
 |\partial_k \partial_l g_{ij}| &\leq \frac{c}{|x|^{p+2}},&
 |R_g| &\leq \frac{c}{|x|^q},
\end{align*}
for $i,j,k,l=1,\ldots,n$, where $|x|=\sqrt{(x^1)^2 + \ldots + (x^n)^2}$, and  $c>0$, $p > \frac{n-2}{2}$, and $q > n$ are constants,
$\delta_{ij}$ is the Kronecker delta, and $R_g$ is the scalar curvature of $g$.
\end{enumerate}
\end{definition}

Next, we recall the definition of the ADM mass \cite{adm}, a number associated
to any asymptotically flat manifold, which in a sense provides a measure of the rate at which the metric decays at infinity.  Bartnik \cite{bartnik} and Chru\'sciel \cite{Chr} proved that the ADM mass is a geometric invariant.
\begin{definition}
The \textbf{ADM mass} of an asymptotically flat manifold $(M,g)$ of dimension $n$ is the number 
$$m_{ADM}(g) = \frac{1}{2(n-1)\omega_{n-1}} \lim_{r\to \infty} \sum_{i,j=1}^n\int_{S_r} \left(\partial_i g_{ij} - \partial _j g_{ii}\right)\frac{x^j}{r} dA$$
where $(x^i)$ are asymptotically flat coordinates, $S_r$ is the coordinate sphere $\{|x|=r\}$, $dA$ is the area form on $S_r$ induced by $\delta_{ij}$, and $\omega_{n-1}$ is the
area of the unit sphere in $\R^n$.
\end{definition}

We recall some other terminology here as well. Let $\Omega \subset \R^n$ be a bounded open set that is smooth, i.e. $\partial \Omega$ is a smooth hypersurface. Recall that $\Omega$ is \emph{mean-convex} if $\partial \Omega$ has nonnegative mean curvature in
the direction pointing into $\R^n \setminus \Omega$. In this paper, any reference to mean-convexity is with respect to the Euclidean metric.
Recall that $\Omega$ is \emph{star-shaped} if there exists a point $x_0 \in \Omega$ such that for each
$x \in \Omega$, the line segment from $x$ to $x_0$ is contained in $\Omega$.

\section{Inequalities for black holes}
\label{sec_bh}

\subsection{The conformally flat case}
Here we prove Theorem \ref{thm_main} from the introduction,  a 
Penrose-like inequality in all dimensions for conformally flat, asymptotically flat metrics.

Note that the class of metrics considered in Theorem \ref{thm_main} includes, for instance, the 
Schwarzschild metric $g_m$ of mass $m>0$: choose $\Omega$ to be the open ball of radius
$\left(\frac{m}{2}\right)^{\frac{1}{n-2}}$ about the origin, and define on 
$M = \R^n \sm \Omega$:
\begin{equation}
g_m = \left(1+\frac{m}{2|x|^{n-2}}\right)^{\frac{4}{n-2}} \delta. \label{eqn_schwarzschild}
\end{equation}

Before presenting the proof of Theorem \ref{thm_main}, we state some consequences:
\begin{cor}
\label{cor_main}
Under the hypotheses of Theorem \ref{thm_main}, we have the following Penrose-like inequalities
\begin{align}
m_{ADM}(g) > \left(\frac{A}{\omega_{n-1}}\right)^{\frac{n-2}{n-1}} \label{ineq_area}\\
m_{ADM}(g) > \left(\frac{V}{\beta_n}\right)^{\frac{n-2}{n}} \label{ineq_vol},
\end{align}
where $A$ is the Euclidean area of $\partial \Omega$, $V$ is the Euclidean volume of $\Omega$, and $\beta_n$
is the Euclidean volume of the unit ball in $\R^n$.
\end{cor}

In the course of the proofs of Theorem \ref{thm_main} and Corollary \ref{cor_main}, we will see that (\ref{ineq_main}), (\ref{ineq_area}), and (\ref{ineq_vol}) are 
not sharp, off by a factor of two in the model case in which $\Omega$ is a round $n$-ball and $g$ is a 
Schwarzschild metric. In all cases, the Euclidean area of $\partial \Omega$ will be strictly less than the $g$-area.

\begin{remark}
Inequality (\ref{ineq_main}) in Theorem \ref{thm_main} does not follow from the RPI in any dimension, even without the scalar curvature integral term. If $\Sigma$ happens to be area-outer-minimizing in $(M,g)$, then
inequalities (\ref{ineq_area}) and (\ref{ineq_vol}) follow from the 
RPI, but only in the dimensions $3 \leq n \leq 7$ for which the RPI is currently known. 
Inequality (\ref{ineq_vol}) is the volumetric Penrose inequality that was proved by Schwartz without the star-convexity
assumption \cite{schwartz}, later improved by Freire--Schwartz (cf. \cite{FS}).
\end{remark}

\begin{proof}[Proof of Theorem \ref{thm_main}]
The scalar curvature $R_g$ of the conformally flat metric $g = u^{\frac{4}{n-2}} \delta$ is given by
the formula
\begin{equation}
\label{eq_conf_scalar_curv}
R_g =u^{-\frac{n+2}{n-2}} \left(  -\frac{4(n-1)}{n-2}\Delta u + R_\delta u\right),
\end{equation}
where $\Delta = \Div \grad$ is the Euclidean Laplacian and $R_\delta=0$ is the scalar curvature of $\delta$.  Therefore the hypothesis of nonnegative scalar 
curvature translates to $\Delta u \leq 0$.  Apply the divergence theorem
over the region $B_r \supset \Omega$ in $M$ bounded by a large coordinate sphere $S_r$ of radius $r$:
$$\int_{B_r \setminus \Omega} \Delta u dV = \int_{S_r} \partial_{\nu}(u)dA - \int_{\partial M} \partial_{\nu}(u) dA,$$
where $dV$, $dA$ and $\nu$ are the volume form, area form, and unit normal with respect to the Euclidean metric.  
In both cases, $\nu$ is chosen to point toward infinity, and $\partial_{\nu}$ is the directional 
derivative with respect to $\nu$.  Since $R_g$ is integrable with respect to $g$
(from $q > n$ in the definition of asymptotic flatness), equation
(\ref{eq_conf_scalar_curv}) and the fact $u \to 1$ at infinity proves that $\Delta u$ is integrable
with respect to $\delta$.  Thus, 
\begin{equation}
\label{nu_u}
-\frac{2}{(n-2)\omega_{n-1}}\lim_{r \to \infty}\int_{S_r} \partial_{\nu}(u) dA 
= -\frac{2}{(n-2)\omega_{n-1}}\int_{\partial M} \partial_{\nu}(u)  dA
  -\frac{2}{(n-2)\omega_{n-1}}\int_{M} \Delta u  dV.
  \end{equation}
It is straightforward to check that the left-hand side is the ADM mass $m$ of $(M,g)$.  Moreover,
equation (\ref{eq_conf_scalar_curv}) and the fact $dV_g=u^{\frac{2n}{n-2}} dV$ reduce this to:
\begin{equation}
\label{eqn1}
m = -\frac{2}{(n-2)\omega_{n-1}}\int_{\partial M} \partial_{\nu}(u) dA + \frac{1}{2(n-1)\omega_{n-1}}
\int_M R_g u\inv dV_g.
\end{equation}
Using the conformal transformation law for the mean curvature and the 
hypothesis that $\partial M$ has zero mean curvature with respect to $g$, we have:
\begin{equation*}
0= u^{-\frac{n}{n-2}} \left(Hu + \frac{2(n-1)}{n-2} \partial_{\nu}(u)\right) \qquad \text{on } \partial M
\end{equation*}
where $H\geq0$ is the Euclidean mean curvature of $\partial \Omega$; this may be rearranged to
\begin{equation}
\label{eqn2}
-\frac{2}{(n-2)\omega_{n-1}} \partial_{\nu}(u) = \frac{1}{(n-1) \omega_{n-1}} Hu.
\end{equation}
Combining (\ref{eqn1}) and (\ref{eqn2}), we obtain:
\begin{equation}
\label{eqn3}
m- \frac{1}{2(n-1)\omega_{n-1}}
\int_M R_g u\inv dV_g= \frac{1}{(n-1)\omega_{n-1}} \int_{\partial M} Hu dA > \frac{1}{(n-1)\omega_{n-1}} \int_{\partial M} H dA,
\end{equation}
having used the fact $u\geq 1$, by the maximum principle. Indeed, $u$ is a superharmonic 
function that approaches one at infinity, with $\partial_{\nu}(u) \leq 0$ on $\partial M$ (cf. Lemma 11 of \cite{schwartz}). 
Strict inequality above holds for the following reason. If not, then $H(u-1)$ is identically zero on $\partial M$. There are no compact minimal hypersurfaces in Euclidean space, so $H$ is strictly positive at some point $p \in \partial M$. Then $u(p)=1$. Since $u$ achieves its global minimum at $p$, then $\partial_\nu(u)(p) > 0$ by the maximum principle or else $u \equiv 1$, either of which is a contradiction.

We apply the Minkowski inequality relating the integral of the mean curvature of $\partial \Omega$ to its area\footnote{Other names for this type of estimate 
are the Aleksandrov--Fenchel inequality and the isoperimetric inequality for quermassintegrals \cite{schneider}.  
Minkowski gave the first proof, for convex regions in $\R^3$ \cite{isoperimetric}.}.  
Let $\Omega_1, \ldots,
\Omega_k$ be the connected components of $\Omega$ with boundaries $\Sigma_1, \ldots, \Sigma_k$ of Euclidean areas $A_1, \ldots, A_k$.  Since
we assume $\Omega_i$ is mean-convex and star-shaped, Guan and Li's proof of the Minkowski inequality \cite{guan_li} applies to show
\begin{equation}
\label{ineq_mink}
\frac{1}{(n-1)\omega_{n-1}}\int_{\Sigma_i} H dA \geq \left(\frac{A_i}{\omega_{n-1}}\right)^{\frac{n-2}{n-1}},
\end{equation}
for each $i=1,\ldots,k$.  Combining this with (\ref{eqn3}) and noting $\partial M = \Sigma_1 \cup \ldots \cup \Sigma_k$, we have
\begin{equation*}
m- \frac{1}{2(n-1)\omega_{n-1}}
\int_M R_g u\inv dV_g > \left(\frac{A_1}{\omega_{n-1}}\right)^{\frac{n-2}{n-1}} + \ldots
	+ \left(\frac{A_k}{\omega_{n-1}}\right)^{\frac{n-2}{n-1}}. \qedhere
\end{equation*}
\end{proof}

\begin{remark}
There is precedent for using the Minkowski inequality to prove Penrose-like inequalities; for instance,
see the work of Gibbons on collapsing shells \cite{gibbons} and Lam on the case of graphs over 
$\R^n \sm \Omega$ \cite{lam}.
\end{remark}

\begin{proof}[Proof of Corollary \ref{cor_main}]
Inequality (\ref{ineq_area}) is immediate, since the integral is nonnegative, and 
$$\left(\frac{A_1}{\omega_{n-1}}\right)^{\frac{n-2}{n-1}} + \ldots
	+ \left(\frac{A_k}{\omega_{n-1}}\right)^{\frac{n-2}{n-1}} \geq
	 \left(\frac{A_1+\ldots+A_k}{\omega_{n-1}}\right)^{\frac{n-2}{n-1}}.$$
To prove (\ref{ineq_vol}), simply follow (\ref{ineq_area}) by the classical isoperimetric inequality \cite{isoperimetric} for $\Omega$:
\begin{equation*}
\left(\frac{A}{\omega_{n-1}}\right)^{\frac{1}{n-1}} \geq \left(\frac{V}{\beta_n}\right)^{\frac{1}{n}}.\qedhere
\end{equation*}
\end{proof}

We conclude this section by remarking that the star-shaped hypothesis in Theorem \ref{thm_main} and Corollary \ref{cor_main} can be removed. Huisken announced the Minkowski inequality (\ref{ineq_mink}) for outward-minimizing domains in $\R^n$ (which are mean-convex) (see \cite{guan_li}); the mean-convex case was proved by Freire and Schwartz \cite{FS}.

\subsection{Example: several distant black holes}
\label{sec_example}
We now construct an example to demonstrate that it is possible for Theorem \ref{thm_main} to give a better (larger) lower bound for the ADM mass than the RPI itself. From a physical point of view, this example will consist of several nearly-Schwarzschild black holes that are mutually far apart.

We restrict to dimension three, although the construction will work for dimension $3 \leq n \leq 7$. Fix a parameter $m>0$ to represent the mass of each black hole, and fix $\epsilon \in (0, \frac{1}{10})$. Fix a positive integer $k$ to represent the number of black holes. Given distinct points $\vec c_1, \ldots, \vec c_k \in \R^3$ (to be specified), define
\begin{align*}
u_i(\vec x) &= \frac{m}{2|\vec x - \vec c_i|}
\end{align*}
for $i=1, \ldots, k$, a positive harmonic function on $\R^3 \setminus \{\vec c_i\}$ that blows up near $\vec c_i$ and approaches zero at infinity.
The Riemannian metric $g_i=(1+u_i)^4 \delta$ on  $\R^3 \setminus \{\vec c_i\}$ is isometric to the two-ended Schwarzschild metric of mass $m$, shown on the left in Figure \ref{fig_schwarz}. 
\begin{figure}[ht]
\caption{Two-ended Schwarzschild manifold}
\begin{center}
\includegraphics[scale=0.6]{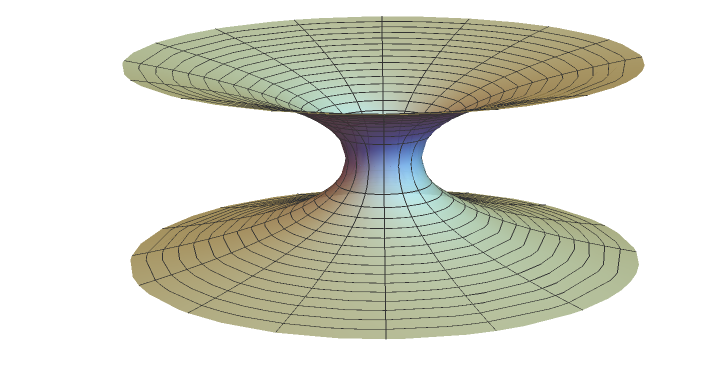}
\includegraphics[scale=0.6]{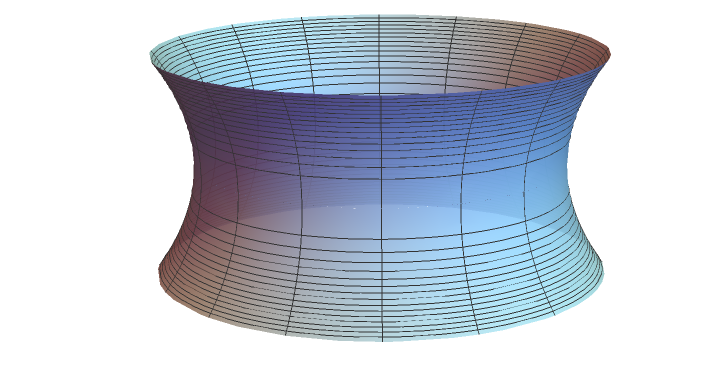}
\end{center}
\flushleft\footnotesize{On the left is a depiction of the Schwarzschild manifold of positive mass. The minimal ``neck'' is an area-minimizing 2-sphere, which we denote by $S_i$, fixed by a reflection symmetry.  On the right is a tubular neighborhood $U_i$ of $S_i$, with $\partial U_i$ of positive mean curvature pointing out of $U_i$.}
\label{fig_schwarz}
\end{figure}
In particular $(\R^3 \setminus \{\vec c_i\}, g_i)$ has a minimal surface $S_i$ located at the 
sphere of Euclidean radius $\frac{m}{2}$ about $\vec c_i$.  $S_i$ has Euclidean area $\pi m^2$ and $g_i$-area  $16\pi m^2$. Note that $S_i$ is contained in a tubular neighborhood $U_i$ in $\R^3$ (shown on the right in Figure \ref{fig_schwarz}) for which $\partial U_i$ has two components, each with  positive mean curvature with respect to $g_i$ in the direction pointing out of $U_i$. Realize $U_i$ as a Euclidean annulus about $\vec c_i$ with inner radius $< \frac{m}{2}$ and outer radius $>\frac{m}{2}$. By choosing the inner radius arbitrarily close to $\frac{m}{2}$, we can arrange that the minimum Euclidean area of surfaces homologous to $S_i$ in $U_i$ is at least $\pi m^2 (1+\epsilon)^{-1}$. Note also that $S_i$ minimizes $g_i$-area in its homology class in $U_i$.

We now describe the construction for $k=2$ before moving on to the general case. Fix $\vec c_1 \in \R^3$. We may choose $\vec c_2$ sufficiently far from $\vec c_1$ so that the $C^1$ norms of $u_1$ on $\ol{U_2}$ and of $u_2$ on $\ol{U_1}$ are arbitrarily small, and that $\ol{U_1}$ and $\ol{U_2}$ are disjoint. In particular, we may choose $\vec c_2$ so that the Riemannian metric on $\R^3 \setminus \{\vec c_1, \vec c_2\}$ given by
$$h = (1+u_1 + u_2)^4 \delta$$
satisfies the following properties:
\begin{itemize}
\item $\partial U_1$ and $\partial U_2$ have positive mean curvature in the directions pointing out of $U_1$ and $U_2$, with respect to $h$.
\item The measurement of hypersurface areas on $U_i$ with respect to $g_i$ and with respect to $h$ differ by a factor of at most $(1+\epsilon)$.
\end{itemize}
The next step is to argue that $U_1$ and $U_2$ each contain a minimal surface with respect to $h$, and the Euclidean and $h$-areas thereof are close to those of $S_i$. Consider the problem of finding an area-minimizer in $(\ol{U_i}, h)$ in the homology class of $S_i$. By standard arguments in geometric measure theory, there exists a smooth area-minimizer, call it $\tilde S_i$, in $U_i$, that is a minimal surface with respect to $h$.\footnote{These results are essentially due to Federer and Fleming \cites{fed_flem, fleming, federer}.
See also the appendix of \cite{schoen_yau}.} (This conclusion relies on the ambient dimension being less than eight, and on $\partial U_i$ having positive mean curvature in the direction pointing out of $U_i$.) Using the second bullet point above, we have
$$|\tilde S_i|_{h} \leq |S_i|_{h} \leq (1+\epsilon) |S_i|_{g_i} = 16\pi m^2 (1+\epsilon).$$
By the choice of $U_i$, the Euclidean area of $\tilde S_i$ must be at least $\pi m^2 (1+\epsilon)^{-1}$. The manifold $M$ we arrive at is $\R^3$ minus the open regions enclosed by $\tilde S_1$ and $\tilde S_2$, equipped with the Riemannian metric $h$. This is an asymptotically flat manifold, with zero scalar curvature (since $1+u_1+u_2)$ is harmonic), with minimal surface boundary $\tilde S_1 \cup \tilde S_2$. As noted after equation (\ref{nu_u}), the flux integral $-\frac{2}{(n-2)\omega_{n-1}}\displaystyle\lim_{r \to \infty}\int_{S_r} \partial_{\nu}(1+u_1+u_2) dA $ gives the ADM mass of $(M,h)$, which in this case in $2m$.

Now we carry out the construction for $k\geq 2$. Fix $\vec c_1 \in \R^3$. We may choose $\vec c_2, \ldots, \vec c_k$ all sufficiently far apart from each other and from $\vec c_1$ so that the $C^1$ norm of $u_1 + \ldots +\widehat{u_j}+\ldots+u_k$ (i.e., with $u_j$ omitted) on $\ol{U_j}$ is arbitrarily small, for each $j=1,\ldots, k$, and the $\ol{U_i}$ are pairwise disjoint.  In particular, we can choose $\vec c_i$ so that 
$$h= (1+u_1 + \ldots + u_k)^4 \delta,$$
a Riemannian metric on $\R^3 \setminus \{ \vec c_1, \ldots, \vec c_k\}$, satisfies
\begin{itemize}
\item $\partial U_i$ has positive mean curvature in direction pointing out of $U_i$ with respect to $h$.
\item The measurement of hypersurface areas on $U_i$ with respect to $g_i$ and with respect to $h$ differ by a factor of at most $1+\epsilon$.
\end{itemize}

The same argument applies to construct minimal surfaces $\tilde S_i$ in $U_i$, each of whose $h$-areas is at most $16\pi m^2 (1+\epsilon)$ and whose Euclidean areas is at least $\pi m^2(1+\epsilon)^{-1}$.
Let $M$ be the manifold obtained by removing from $\R^3$ the open regions bounded by the $k$ different $\tilde S_i$. Then $(M,h)$ is asymptotically flat, with zero scalar curvature, with ADM mass $km$, and with boundary $\displaystyle \cup_{i=1}^k \tilde S_i$ consisting of minimal surfaces.

We now compare the lower bounds given by the RPI and by Theorem \ref{thm_main} for $(M,h)$. Let $A$ be the total $h$-area of the minimal surfaces $\partial M$. If $\partial M$ is area-outer-minimizing in $(M,h)$, then the RPI gives a lower bound of $\frac{1}{2}\sqrt{\frac{A}{4\pi}}$ for the ADM of $(M,g)$. If not, then by the aforementioned results from geometric measure theory, there exists an area-outer-minimizing minimal surface (possibly disconnected) enclosing $\partial M$, of area at most $A$. The RPI would then apply to the manifold outside this minimal surface, with ADM mass lower bound at most $\frac{1}{2}\sqrt{\frac{A}{4\pi}}$. Thus, in all cases, the RPI's lower bound is at most
\begin{equation}
\label{eqn_lb1}
\frac{1}{2}\sqrt{\frac{16\pi m^2(1+\epsilon) k}{4\pi}} =  m \sqrt{1+\epsilon}\sqrt{k}.
\end{equation}
On the other hand, the lower bound in Theorem \ref{thm_main} is at least
\begin{equation}
\label{eqn_lb2}
k\sqrt{\frac{ \pi m^2}{4\pi (1+\epsilon)}} = \frac{ m k}{2\sqrt{1+\epsilon}}.
\end{equation}
In particular, the growth rate of the lower bound in Theorem \ref{thm_main} is proportional to the number of black holes $k$, rather than $\sqrt{k}$ as in the RPI. In fact, for $k\geq 5$, the original choice $\epsilon < \frac{1}{10}$ guarantees  (\ref{eqn_lb2}) exceeds (\ref{eqn_lb1}).

\subsection{Beyond the conformally flat case}
The following theorem is a generalization of Theorem \ref{thm_main} in which a scalar-flat background metric
$\ol g$ replaces the flat metric, $\delta$.
\begin{thm}
\label{thm_general}
Suppose $M = \R^n \sm \Omega$ for a smooth, mean-convex, bounded open set $\Omega$, of which every connected
component is star-shaped.  Let $\ol g$ be any asymptotically flat metric on $M$ with the 
following properties:
\begin{enumerate}
 \item[(i)] $\ol g$ has zero scalar curvature,
 \item[(ii)] $\ol g$ and $\delta$ agree on $\partial M$, and
 \item[(iii)] $\ol g$ and $\delta$ induce the same mean curvature on $\partial M$.
\end{enumerate}
Let $g = u^{\frac{4}{n-2}} \ol g$, where $u$ is a smooth, positive function tending to one at infinity.
Assume that $g$ is asymptotically flat, has nonnegative scalar curvature, and $\partial M = \partial \Omega$ 
has zero mean curvature in $(M,g)$.  Then:
\begin{equation}
\label{ineq_general}
m_{ADM}(g) > \left(\frac{A_1}{\omega_{n-1}}\right)^{\frac{n-2}{n-1}} + \ldots
	+ \left(\frac{A_k}{\omega_{n-1}}\right)^{\frac{n-2}{n-1}} + \frac{1}{2(n-1)\omega_{n-1}}
\int_M R_g u\inv dV_g,
\end{equation}
where, as before, $A_1, \ldots A_k$ are the Euclidean areas of the components of $\partial \Omega$.
\end{thm}
In other words, we merely assume that $g$ is conformal to a scalar-flat metric agreeing
with the Euclidean metric on $\partial M$ in a suitable first-order sense.  

\begin{proof}
The proof is nearly identical to that of Theorem \ref{thm_main}, with one additional step.  First, note the conformal transformation
laws for the scalar curvature and mean curvature are the same as in the proof of Theorem \ref{thm_main},
with $\delta$ replaced by $\ol g$, and $u$ is consequently $\ol g$-superharmonic.  (Here we are using (i)--(iii).)
The only other issue in extending the proof to this more general case is the following:
the integral of $\partial_{\ol \nu}(u)$ over a coordinate sphere at infinity measures the difference of the ADM masses of $g$ and $\ol g$:
\begin{equation}
\label{eqn_diff_adm}
-\frac{2}{(n-2)\omega_{n-1}}\lim_{r \to \infty}\int_{S_r} \partial_{\ol \nu}(u) \ol{dA} = m_{ADM}(g) - m_{ADM}(\ol g),
\end{equation}
where $\ol \nu$ and $\ol {dA}$ are the unit normal and area form with respect to $\ol g$.  
(Previously, in the case of Theorem \ref{thm_main}, $m_{ADM}(\ol g) = m_{ADM}(\delta)=0$.)

To recycle the proof of Theorem \ref{thm_main}, we need only show that $m_{ADM}(M, \ol g) \geq 0$ (since $\ol g$ and $\delta$ induce the same area and mean curvature on $\partial M$).
Define a Riemannian metric $\tilde g$ on $\R^n$ by gluing the metric $\ol g$
on $\R^n \sm \Omega$ and the metric $\delta$ defined on $\Omega$.  By (ii), $\tilde g$ is Lipschitz across
$\partial \Omega$ 
and is smooth and scalar-flat away from $\partial \Omega$.  Moreover, the mean curvatures of both sides 
of $\partial \Omega$ agree by (iii). By the ``positive mass theorem with corners'' (proved by Shi and Tam for spin manifolds \cite{shi_tam}, of which $\R^n$ is one, and by Miao without the spin assumption \cite{miao}), the ADM mass of $(M, \tilde g)$ 
is nonnegative. The proof is complete, since $\tilde g$ and $\ol g$
have the same ADM mass.
\end{proof}

\begin{remark}
Interestingly, Theorem \ref{thm_general} does not follow from the Riemannian
Penrose inequality in dimensions $3 \leq n \leq 7$, even if we decrease the right-hand side
of (\ref{ineq_general})
to $\left(\frac{A}{\omega_{n-1}}\right)^{\frac{n-2}{n-1}}$, where $A = A_1 + \ldots + A_k$.
The RPI estimates $m_{ADM}(g)$ from below only in terms of any area-outer-minimizing minimal surface in $(M,g)$. If $\partial M$ fails to be area-outer-minimizing, then the RPI is not sensitive to the area of $\partial M$. Further discussion along these lines is given in section \ref{sec_discuss}.
\end{remark}

\subsection{Removing the hypothesis on scalar curvature}
In Theorems \ref{thm_main} and \ref{thm_general}, the nonnegativity of scalar curvature implied superharmonicity of the conformal factor
$u$, which in turn implied $u \geq 1$.  In the next theorem, we derive a similar inequality without assuming
nonnegative scalar curvature.

\begin{thm}
\label{thm_general2}
Suppose $(M,g)$ is as in Theorem \ref{thm_general}, except remove the assumption that $g$ has nonnegative
scalar curvature.  Then
\begin{equation*}
m_{ADM}(g) > \left(\frac{A_1}{\omega_{n-1}}\right)^{\frac{n-2}{n-1}} + \ldots
	+ \left(\frac{A_k}{\omega_{n-1}}\right)^{\frac{n-2}{n-1}} + \frac{1}{2(n-1)\omega_{n-1}}
\int_M R_g u^{-2} dV_g.
\end{equation*}
\end{thm}
Note the weighting factor on the integral of scalar curvature is $u^{-2}$ rather than $u\inv$ as before, and this
integral may have any sign.  Penrose inequalities including a weighted integral of scalar curvature may be of interest in proving
the conjectured general Penrose inequality for slices of spacetimes that are not totally geodesic 
\cite{bray_khuri}.

\begin{proof}
Let $\ol g, u,$ and $g$ be as in the statement of Theorem \ref{thm_general}, except without the assumption
that $g$ has nonnegative scalar curvature.  
Since $u \to 1$ at infinity, $u>0$, and $\partial M$ is compact, $u$ is bounded above and below by positive constants.  

Next, since $g$ and $\ol g$ are both asymptotically flat, their scalar curvatures are both integrable with respect to the respective metrics. Since $u \to 1$ at infinity, their scalar curvatures are both integrable with respect to $\ol g$. Then as in (\ref{eq_conf_scalar_curv}), we have that $\frac{\ol \Delta u}{u}$ is integrable on $(M, \ol g)$. Also, since $g$ and $\ol g$ are asymptotically flat, we see that $|\ol \nabla u|_{\ol g}$ is $O(r^{-p-1})$, for $p>\frac{n-2}{2}$, where $\ol \nabla$ is the gradient with respect to $\ol g$. Then $\frac{|\ol \nabla u|^2_{\ol g}}{u^2}$ is $O(r^{-2p-2})$ and hence is integrable on $(M, \ol g)$, as $2p+2>n$. Thus
$$\ol{\div}\left(\frac{\ol \nabla u}{u}\right)=\frac{\ol \Delta u}{u} - \frac{|\ol \nabla u|^2_{\ol g}}{u^2}$$
is integrable on $(M,\ol g)$, where $\ol \div$ is $\ol g$-divergence. Integrating by parts yields
$$\lim_{r \to \infty} \int_{S_r} \frac{\partial_{\ol \nu}(u)}{u} \ol {dA} 
- \int_{\partial M} \frac{\partial_{\ol \nu}(u)}{u} \ol{dA} = \int_M \frac{\ol \Delta u}{u} \ol {dV} 
  -\int_M \frac{|\ol \nabla u|^2_{\ol g}}{u^2} \ol {dV}$$
where $S_r$ is the coordinate sphere of radius $r$, and $\ol{dV}$, $\ol{dA}$ and $\ol \nu$ are the volume form,
area form, and unit normal with respect to $\ol g$.  In the first term above, the $u$ in the denominator may be ignored, since $u \to 1$ at infinity.  
Multiplying by $-\frac{2}{(n-2)\omega_{n-1}}$, applying the conformal transformation laws for the ADM
mass (\ref{eqn_diff_adm}), mean curvature (\ref{eqn2}), scalar curvature (\ref{eq_conf_scalar_curv}), and volume form, and discarding the signed $|\ol \nabla u|^2$ term yields 
the inequality:
$$m_{ADM}(g) - m_{ADM}(\ol g) > \frac{1}{(n-1)\omega_{n-1}}\int_{\Sigma} \ol H \ol {dA} + \frac{1}{2(n-1)\omega_{n-1}}
\int_M R_g u^{-2} dV_g,$$
where $\ol H$ is the mean curvature of $\Sigma$ with respect to $\ol g$.  The hypotheses 
on $\ol g$ and $\ol H$ along $\partial M$ allow us to 
replace $\ol H \ol {dA}$ with $H dA$.  To complete the proof, apply the Minkowski inequality to $\int_\Sigma H dA$
as in the proof of Theorem \ref{thm_main}, and use the nonnegativity of the ADM mass of $\ol g$ as in
the proof of Theorem \ref{thm_general}.
\end{proof}

\subsection{Discussion of lack of area-outer-minimizing hypothesis}
\label{sec_discuss}
In the Riemannian Penrose inequality (Theorem \ref{thm_rpi}), it is well-known that the 
hypothesis that the boundary $\Sigma$ be area-outer-minimizing is crucial.  Indeed, 
one may easily construct rotationally symmetric, asymptotically flat manifolds $(M,g)$ of nonnegative scalar
curvature and minimal boundary $\Sigma$ of area $A$ such that the ratio $m_{ADM}(g)/A^{\frac{n-2}{n-1}}$ is arbitrarily
small.  In this case, $\Sigma$ is ``hidden'' behind some area-outer-minimizing minimal surface 
$\tilde \Sigma$ of area $\tilde A$, and the Riemannian Penrose inequality holds for the region 
exterior to $\tilde \Sigma$ (see figure 1 of \cite{imcf}).

The results of this paper do not require the boundary to be area-outer-minimizing, which is perhaps philosophically
the reason why we do not recover a sharp version of the RPI for conformally flat manifolds.
This phenomenon (together with the difficulty of utilizing the area-outer-minimizing hypothesis) was pointed out
by Bray and Iga \cite{bray_iga}.  Nevertheless, rotationally symmetric manifolds are conformally flat, 
so the above examples of hidden minimal surfaces make it interesting that we can prove any such inequality at all without an area-outer-minimizing hypothesis.

\section{Inequalities for zero area singularities}
\label{sec_zas}
We now recall the idea of a \emph{zero area singularity}, which is in a sense dual to the idea of a black hole manifesting 
as a minimal surface.  Suppose $M$ is a smooth manifold with smooth compact boundary $\partial M$.  Let $g$ be a smooth Riemannian metric defined on the interior $M \sm \partial M$.
A connected component $S$ of $\partial M$ is said to be a \emph{zero area singularity} (ZAS) of $g$ if for 
all sequences $\{S_i\}$ of hypersurfaces in $M \sm \partial M$ converging in the $C^1$ sense to $S$, we have
$$\lim_{i \to \infty} |S_i|_g = 0,$$
where $|S_i|_g$ is the area of $S_i$ with respect to $g$.  (Note that $C^1$ convergence depends only on the smooth structure of $M$.)  The study of ZAS was initiated by Bray \cite{bray_npms}.  For more details and precise definitions, see \cite{zas} and the survey paper \cite{bray_survey}. Other work on ZAS was carried out by Robbins \cite{zas_robbins} and also in \cite{thesis}.

The motivating example of a manifold containing a zero area singularity is the Schwarz\-schild manifold of negative mass,
described as follows.  For $n \geq 3$ and a real parameter $m < 0$, suppose $M$ is $\R^n$ minus the open ball about 
the origin of radius
$\left(\frac{|m|}{2}\right)^{\frac{1}{n-2}}.$  As in (\ref{eqn_schwarzschild}), let $g_m$ be the metric on $M$ defined by
$$g_m = \left(1+\frac{m}{2|x|^{n-2}}\right)^{\frac{4}{n-2}} \delta.$$
It is not difficult to see that the boundary sphere $\partial M$ is a ZAS of the metric $g_m$, since the conformal factor vanishes on $\partial M$.

In the context of this paper, it is natural to restrict to the class of metrics containing \emph{regular} ZAS.
This means that on an open set $U$ containing a ZAS boundary component $S$, the metric is given by
$g = u^{\frac{4}{n-2}} \ol g$, where $\ol g$ is a smooth Riemannian metric on $U$ (up to and including $S$) 
and $u\geq 0$ is a smooth function on $U$ vanishing only on $S$, with
$\partial_{\ol \nu} u > 0$ on $S$.  
(Here $\ol \nu$ is the unit normal to $S$ with respect to $\ol g$, pointing into the manifold.)  
For instance, the singularity in the Schwarzschild manifold of negative mass is a regular ZAS,
with $\ol g = \delta$.

Analogous to the definition of the mass of a collection of black holes of total area $A$ to be
$\frac{1}{2} \left(\frac{A}{\omega_{n-1}}\right)^{\frac{n-2}{n-1}}$, Bray proposed the following formula
to define the mass of a collection of regular ZAS $\Sigma$ of a metric $g = u^{\frac{4}{n-2}} \ol g$:
$$m_{ZAS}(\Sigma) = -\frac{2}{(n-2)^2}\left(\frac{1}{\omega_{n-1}}\int_\Sigma (\partial_{\ol \nu} u)^{\frac{2(n-1)}{n}} \ol{dA} \right)^{\frac{n}{n-1}},$$
where $\ol {dA}$ is the area measure on $\Sigma$ with respect to $\ol g$ \cites{bray_npms,zas,bray_survey}.  This negative real number depends only on the local geometry of $g$ near $\Sigma$ (not on the pair $(\ol g,u)$) and produces the value $m$ for the ZAS in the Schwarzschild
manifold of mass $m<0$ \cite{zas}. 

Motivated by the Riemannian Penrose inequality, Bray conjectured that in an asymptotically flat
manifold $(M,g)$ of nonnegative scalar curvature for which every component of the compact boundary 
$\Sigma = \partial M$ is a ZAS of $g$, the ADM mass ought to be bounded below in terms of the ZAS mass:
\begin{equation}
\label{eqn_rzi}
m_{ADM}(g) \geq m_{ZAS}(\Sigma).
\end{equation}
This inequality remains a conjecture.  Some special cases in which (\ref{eqn_rzi}) is known are:
\begin{itemize}
\item $n=3$ and $\Sigma$ is connected (due to Robbins \cite{zas_robbins}, using the inverse mean curvature flow technique of Huisken--Ilmanen \cite{imcf}), 
\item $3 \leq n \leq 7$ and $g = u^{\frac{4}{n-2}} \ol g$,  where $(M, \ol g)$ satisfies the hypothesis of the Riemannian
Penrose inequality (Theorem \ref{thm_rpi}); see \cites{bray_npms, bray_survey, zas, jau, thesis}, or
\item $(M,g)$ is a graph over $\R^n$ in Minkowski space $\R^{n,1}$, 
and each of the ZAS is a level set of the graph function \cite{lam_thesis}.  (But note that Lam's definition of $m_{ZAS}$
in \cite{lam_thesis} is somewhat different.)
\end{itemize}

We emphasize that the positive mass theorem does not apply to manifolds that contain regular ZAS; 
such manifolds are incomplete and may have negative ADM mass.  Therefore (\ref{eqn_rzi}) may be viewed as a 
generalization of the PMT, providing a lower bound for the ADM mass of manifolds that contain certain types of singularities.

\vspace{5mm}
\paragraph{\emph{Setup:} }
Our goal is to prove a version of (\ref{eqn_rzi}) in the conformally flat case.  
For $n \geq 3$, suppose that $\Omega \subset \R^n$ is a smooth bounded open set, such that $M = \R^n \sm \Omega$
is connected.  Let $\Sigma = \partial M$. 
Assume $u$ is a function on $M$ with the following properties:
\begin{enumerate}
 \item[(i)] $u \to 1$ at infinity, $u\geq 0$, and $g=u^{\frac{4}{n-2}} \delta$ is asymptotically flat, 
 \item[(ii)] $u^{-1}(0)=\Sigma$, and
 \item[(iii)] $\Delta u \leq 0$ (equivalently, $g$ has nonnegative scalar curvature).
\end{enumerate}
By the maximum principle, $u > 0$ in the interior of $M$ and $\partial_\nu(u) > 0$ on $\Sigma$, where the Euclidean unit normal $\nu$ to $\Sigma$ points into $M$.  Therefore, each component of $\Sigma$
is a regular ZAS of $g$.

We first give a result that estimates the ADM mass of $(M,g)$ from below in terms of the ZAS mass
and the Euclidean area $A$ of $\Sigma$.  
\begin{lemma} With the above setup,
\label{lemma_zas_lower_bound}
$$m_{ADM}(g) \geq m_{ZAS}(\Sigma) - \frac{1}{2}\left(\frac{A}{\omega_{n-1}}\right)^{\frac{n-2}{n-1}}.$$
\end{lemma}
\begin{proof}
Integrating $\Delta u \leq 0$ over $M$ gives the inequality (cf. the proof of Theorem \ref{thm_main}):
$$m \geq -\frac{2}{(n-2)\omega_{n-1}}\int_{\partial M} \partial_{\nu}(u) dA,$$
where $m=m_{ADM}(g)$. 
By H\"older's inequality, 
\begin{align*}
m &\geq -\frac{2}{(n-2)\omega_{n-1}} \left(\int_\Sigma (\partial_\nu u)^{\frac{2(n-1)}{n}} dA \right)
^{\frac{n}{2(n-1)}} A^{\frac{n-2}{2(n-1)}}\\
  &= \underbrace{\frac{1}{2}\left(\frac{A}{\omega_{n-1}}\right)^{\frac{n-2}{n-1}}-\frac{2}{(n-2)} \left(\frac{1}{\omega_{n-1}}\int_\Sigma (\partial_\nu u)^{\frac{2(n-1)}{n}} dA \right)
^{\frac{n}{2(n-1)}} \left(\frac{A}{\omega_{n-1}}\right)^{\frac{n-2}{2(n-1)}}} - \frac{1}{2}\left(\frac{A}{\omega_{n-1}}\right)^{\frac{n-2}{n-1}}.
\end{align*}
We invoke an argument of Bray \cites{bray_npms,zas}, viewing the underbraced terms as a degree-two polynomial $\frac{1}{2}x^2-bx$ in the 
variable $x = \left(\frac{A}{\omega_{n-1}}\right)^{\frac{n-2}{2(n-1)}}$.  Minimizing over $x\in \R$ gives
$$m\geq -\frac{2}{(n-2)^2}\left(\frac{1}{\omega_{n-1}}\int_\Sigma (\partial_\nu u)^{\frac{2(n-1)}{n}} dA \right)
^{\frac{n}{n-1}}- \frac{1}{2}\left(\frac{A}{\omega_{n-1}}\right)^{\frac{n-2}{n-1}}.$$
Recognizing the first term on the right-hand side as $m_{ZAS}(\Sigma)$ completes the proof.
\end{proof}

The estimate in Lemma \ref{lemma_zas_lower_bound} is unsatisfactory for the reason that the error term 
is additive rather than multiplicative. The following theorem provides a remedy.
\begin{thm} With the above setup,
\label{thm_zas}
$$m_{ADM}(g) \geq m_{ZAS}(\Sigma) \left(1+\frac{1}{4}\iota^2\right),$$
where 
$$\iota = \left.\left(\frac{A}{\omega_{n-1}}\right)^{\frac{n-2}{n-1}} \middle/ \left(\frac{V}{\beta_n}\right)^{\frac{n-2}{n}}\right.$$
is the isoperimetric ratio of $\Omega$.  Here, $A$ and $V$ are the Euclidean area and volume of 
$\partial \Omega$ and $\Omega$.
\end{thm}
Recall the classical isoperimetric inequality is the statement $\iota \geq 1$ \cite{isoperimetric}.

\remark Since $m_{ZAS}(\Sigma)<0$, the estimate of Theorem \ref{thm_zas} is weaker than the conjectured
inequality (\ref{eqn_rzi}). In the model case in which $\Omega$ is a round ball and $u$ is harmonic (or equivalently,
$(M,g)$ is a negative-mass Schwarzschild manifold), the inequality is suboptimal by a factor of $\frac{5}{4}$. Nevertheless, no general version of (\ref{eqn_rzi}) is known when more than one ZAS is present, even in dimension three.

The idea of the proof is to use Lemma \ref{lemma_zas_lower_bound} in conjunction with an upper bound 
on the area $A$ in terms of the isoperimetric ratio of $\Omega$ and the absolute value of the ZAS mass.  Before
proceeding, recall that the \emph{capacity} of the bounded open set $\Omega \subset \R^n$ is the number:
\begin{equation}
\label{eqn_capac}
\capac(\Omega) = \frac{1}{(n-2)\omega_{n-1}} \inf_\psi \left\{\int_{\R^n \sm \Omega} |\nabla \psi|^2 dV\right\},
\end{equation}
where the infimum is taken over all smooth functions $\psi$ on $\R^n \sm \Omega$ that vanish on $\partial \Omega$
and approach one at infinity.  (The geometric quantities in this expression are taken with respect to the Euclidean
metric.)  The infimum is attained by the unique harmonic
function $\varphi$ that vanishes on $\Sigma=\partial \Omega$ and approaches one at infinity. The value of the capacity of $\Omega$ can also be written as:
\begin{align}
\capac(\Omega) &= \frac{1}{(n-2)\omega_{n-1}} \int_{\R^n \sm \Omega} |\nabla \varphi|^2 dV && \text{($\varphi$ achieves infimum)}\nonumber\\
&= \frac{1}{(n-2)\omega_{n-1}} \int_{\R^n \sm \Omega} \Div(\varphi \nabla \varphi) - \cancel{\varphi \Delta \varphi} dV&& \text{(identity; $\Delta \varphi=0$)}\nonumber\\
&= \frac{1}{(n-2)\omega_{n-1}} \left( \lim_{r \to \infty} \int_{S_r} \varphi \partial_\nu \varphi dA - \cancel{  \int_{\Sigma} \varphi \partial_\nu \varphi  dA}\right)&& \text{(divergence theorem; $\varphi|_{\Sigma}=0$)}\nonumber\\
&= \frac{1}{(n-2)\omega_{n-1}} \left( \lim_{r \to \infty} \int_{S_r} \partial_\nu \varphi dA\right)&& \text{($\varphi \to 1$ at infinity)}\nonumber\\
&=\frac{1}{(n-2)\omega_{n-1}} \int_{\Sigma} \partial_\nu\varphi dA&& \text{(divergence theorem; $\Delta \varphi=0$).} \label{eqn_capac2}
\end{align}
In these calculations, the unit normal $\nu$ always points toward infinity.

\begin{proof}[Proof of Theorem \ref{thm_zas}]
Let $\varphi$ be as above.
By the superharmonicity of $u$ and the maximum principle applied to $u-\varphi$, we have $\partial_\nu \varphi \leq \partial_\nu u$
on $\Sigma$, so by (\ref{eqn_capac2}):
$$\capac(\Omega) \leq \frac{1}{(n-2)\omega_{n-1}} \int_\Sigma \partial_\nu u dA.$$
By H\"{o}lder's inequality,
$$\capac(\Omega) \leq \frac{1}{(n-2)\omega_{n-1}} \left(\int_\Sigma (\partial_\nu u)^{\frac{2(n-1)}{n}} dA \right)
^{\frac{n}{2(n-1)}} A^{\frac{n-2}{2(n-1)}}.$$
The inequality of Poincar\'e--Faber--Szeg\"o \cite{isoperimetric} relating the capacity of a region in $\R^n$ to its volume states (see the appendix herein for a proof):
$$\capac(\Omega) \geq \left(\frac{V}{\beta_n}\right)^{\frac{n-2}{n}}.$$
Combining the last two inequalities and squaring, we have:
$$\left(\frac{V}{\beta_n}\right)^{\frac{2(n-2)}{n}} \leq \frac{1}{(n-2)^2} \left(\frac{1}{\omega_{n-1}}\int_\Sigma (\partial_\nu u)^{\frac{2(n-1)}{n}} dA \right)
^{\frac{n}{n-1}} \left(\frac{A}{\omega_{n-1}}\right)^{\frac{n-2}{n-1}}.$$
Using the definition of ZAS mass and isoperimetric ratio, this becomes:
\begin{equation}
\label{eqn_area_upper_bound}
\left(\frac{A}{\omega_{n-1}}\right)^{\frac{n-2}{n-1}} \leq \frac{\iota^2}{2} |m_{ZAS}(\Sigma)|.
\end{equation}
Inequality (\ref{eqn_area_upper_bound}) and the estimate from Lemma \ref{lemma_zas_lower_bound} imply
the result.
\end{proof}

\begin{remark}
Theorem \ref{thm_zas} does not immediately generalize to background metrics $\ol g$ as in Theorem \ref{thm_general}, for the reason that the capacity--volume inequality depends on
the global Euclidean nature of the region outside $\Omega$.  On the other hand, Theorem \ref{thm_zas} 
requires no star-shaped or mean-convexity assumptions on $\Omega$.
\end{remark}

\section{A mixed inequality for black holes and ZAS}
\label{sec_mixed}
Based on considerations in Newtonian physics pertaining to potential energy, Bray conjectured that 
in an asymptotically flat manifold of nonnegative scalar curvature with boundary $\Sigma$ consisting of area-outer-minimizing minimal
surfaces $\Sigma_+$ (of total area $|\Sigma_+|_g$) and zero area singularities $\Sigma_-$, the ADM mass ought to 
be bounded below as follows \cite{bray_npms}, cf. \cite{zas}:
\begin{equation}
\label{ineq_mixed}
m_{ADM}(g) \geq \frac{1}{2}\left(\frac{|\Sigma_+|_g}{\omega_{n-1}}\right)^{\frac{n-2}{n-1}}
+m_{ZAS}(\Sigma_-).
\end{equation}
Consider this problem in the conformally flat case, with the following setup.  
Suppose that $\Omega_+$ and $\Omega_-$ are smooth bounded open sets in $\R^n$ whose closures do not 
intersect.  Assume that $\Omega_+$ is mean-convex and every connected component is star-shaped.  
Set $\Omega = \Omega_{+} \cup \Omega_{-}$, and assume $M=\R^n \sm \Omega$ is connected.  
Let $\Sigma_{\pm}$ be the connected components of $\partial \Omega_{\pm}$, so 
$\partial M = \Sigma_+ \cup \Sigma_-$.
Let $u \geq 0$ be a smooth function on $M$ with the following properties:
\begin{enumerate}
 \item[(i)] $u^{-1}(0)=\Sigma_-$,
 \item[(ii)] $u \to 1$ at infinity and $g=u^{\frac{4}{n-2}} \delta$ is asymptotically flat,
 \item[(iii)] $\Delta u \leq 0$ (equivalently, $g$ has nonnegative scalar curvature away from $\Sigma_-$), and
 \item[(iv)] $\Sigma_+$ has zero mean curvature with respect to $g$.  Equivalently,
    $$Hu + \frac{2(n-1)}{n-2} \partial_\nu u = 0\qquad\text{ on } \Sigma_+,$$
where $H$ is the mean curvature of $\Sigma_+$ with respect to the Euclidean metric.
\end{enumerate}
In the metric $g$, each component of $\Sigma_+$ is a minimal surface, and each component of $\Sigma_-$
is a regular ZAS.  (Note $\partial_\nu (u) > 0$ on $\Sigma_-$ by the maximum principle.)

We make the \emph{ad hoc} assumption that $u \geq 1$ on $\Sigma_+$ and proceed as follows.  First, integrate 
$\Delta u \leq 0$ over $M$.  Arguments from the proofs of Theorem \ref{thm_main} and Lemma \ref{lemma_zas_lower_bound}, as well as $u \geq 1$,
give the inequalities:
\begin{align*}
m_{ADM}(g) &\geq \frac{1}{(n-1)\omega_{n-1}}\int_{\Sigma_+} Hu dA - \frac{2}{(n-2)\omega_{n-1}}\int_{\Sigma_-} \nu(u) dA\\
 &> \left(\frac{A_+}{\omega_{n-1}}\right)^{\frac{n-2}{n-1}} + m_{ZAS}(\Sigma_-) - \frac{1}{2} \left(\frac{A_-}{\omega_{n-1}}\right)^{\frac{n-2}{n-1}},
 \end{align*}
where $A_{\pm}$ is the Euclidean area of $\Sigma_{\pm}$.  We now essentially apply the same steps
as in the proof of Theorem 
\ref{thm_zas}.  Let $\varphi$ be the $\delta$-harmonic function on $\R^n \setminus \Omega_-$ that vanishes on $\Sigma_-$ and tends to one at infinity.  Restrict $\varphi$ to $M$. Since $u = 0$ on $\Sigma_-$, $u \geq 1$ on $\Sigma_+$ by assumption, and $u$ is superharmonic, we have $\partial_{\nu} \varphi \leq \partial_{\nu} u$
on $\Sigma_-$.  Running through the same argument as in Theorem \ref{thm_zas} gives:
$$m \geq \left(\frac{A_+}{\omega_{n-1}}\right)^{\frac{n-2}{n-1}} + m_{ZAS}(\Sigma_-)\left(1+\frac{1}{4}\iota_-^2\right),$$
where $\iota_-$ is the isoperimetric constant of $\Omega_-$.  This is a weakened version of
 the conjectured inequality (\ref{ineq_mixed}), assuming $u \geq 1$ on $\Sigma_+$.

An interesting problem would be to determine whether $u\geq 1$ on $\Sigma_+$ 
holds for any superharmonic function $u$ with the above boundary conditions. It is possible that the assumption $u \geq 1$ on $\Sigma_+$ is preventing $\Sigma_+$ from being very close to $\Sigma_-$; indeed if these surfaces are close with $u=0$ on $\Sigma_-$, then $u < 1$ on $\Sigma_+$ may be possible.

\section*{Appendix: The Poincar\'e--Faber--Szeg\"o capacity--volume inequality}
For reference, we include a proof of the capacity--volume inequality used in the proof of Theorem
\ref{thm_zas}.  The following is based entirely 
on the dimension three case of \cite{isoperimetric}. In this appendix, all geometric quantities are with respect to the Euclidean metric.

\begin{thm} [Poincar\'e--Faber--Szeg\"o] Let $\Omega$ be a bounded open set in $\R^n$ with smooth boundary
such that $\R^n \sm \Omega$ is connected. Then
\begin{equation}
\label{ineq_cap_vol}
\capac(\Omega) \geq \left(\frac{V}{\beta_n}\right)^{\frac{n-2}{n}},
\end{equation}
where $\capac(\Omega)$ and $V$ are the capacity (see (\ref{eqn_capac})) and volume of $\Omega$, respectively.
\end{thm}
\begin{proof}
Let $0 \leq \varphi<1$ be the unique function on $\R^n \sm \Omega$ that vanishes on $\partial \Omega$, is harmonic on
$\R^n \sm \Omega$, and approaches 1 at infinity.  Then
\begin{align*}
 (n-2)\omega_{n-1} \capac(\Omega) &= \int_{\R^n \sm \Omega} |\nabla \varphi|^2 dV.
\end{align*}
For $t \in [0,1)$, let $\Sigma_t$ be the level set $\varphi\inv(t)$. Note that $\Sigma_t$ is smooth for almost every $t$, and $|\nabla \varphi|\neq 0$ on $\Sigma_t$ for such $t$.
By the co-area
formula,
\begin{equation}
\label{a1}
\int_{\R^n \sm \Omega} |\nabla \varphi|^2 dV = \int_0^1 \int_{\Sigma_t} |\nabla \varphi|^2 \frac{1}{|\nabla \varphi|} dA_t dt,
\end{equation}
where $dA_t$ is the area form on $\Sigma_t$.  By the Schwarz inequality, for almost all $t \in [0,1)$,
\begin{equation}
\label{a2}
|\Sigma_t|^2 \leq \left(\int_{\Sigma_t} |\nabla \varphi| dA_ t\right) \left(\int_{\Sigma_t} \frac{1}{ |\nabla \varphi|} dA_ t\right),
\end{equation}
where $|\Sigma_t|$ is the area of $\Sigma_t$.  Combining (\ref{a1}) and (\ref{a2}) produces:
$$\int_{\R^n \sm \Omega} |\nabla \varphi|^2 dV \geq \int_0^1 \frac{|\Sigma_t|^2}{\int_{\Sigma_t} \frac{1}{|\nabla \varphi|} dA_t} dt.$$
Let $V(t)$ be the volume in $\R^n$ of the region bounded by $\Sigma_t$, so again by the co-area formula:
$$V(t) = \vol(\Omega) + \int_0^t \int_{\Sigma_s} \frac{1}{|\nabla \varphi|} dA_s ds,$$
and therefore
$$V'(t) =\int_{\Sigma_t} \frac{1}{|\nabla \varphi|} dA_t$$
for almost all $t \in [0,1)$.
Combining the above gives
\begin{align*}
 (n-2)\omega_{n-1} \capac(\Omega) &\geq\int_0^1 \frac{|\Sigma_t|^2}{V'(t)} dt\\
  &\geq  \int_0^1 \frac{(\omega_{n-1})^2 \left(\frac{V(t)}{\beta_n}\right)^{\frac{2(n-1)}{n}}}{V'(t)} dt,
\end{align*}
where we have used the isoperimetric inequality on the second line.  
Let $R(t)$ be the radius of the sphere that has volume equal to $V(t)$, i.e., $V(t)=\beta_n R(t)^n$.
Then $V'(t) = n\beta_n R(t)^{n-1} R'(t)$ for almost all $t \in [0,1)$, so 
\begin{equation}
\label{eqn_R_prime}
 (n-2)\omega_{n-1} \capac(\Omega) \geq  
    \int_0^1  \frac{\omega_{n-1}R(t)^{n-1}}{R'(t)} dt,
\end{equation}
having used the fact $n\beta_n = \omega_{n-1}$.

Now, let $\tilde \Omega$ be the open ball
about the origin with the same volume as $\Omega$.  Let $\tilde \Sigma_t$ be the sphere about the origin 
of radius $R(t)$, with area form $d \tilde A_t$. Note $\tilde \Sigma_0 = \partial \tilde \Omega$. Let $\tilde \varphi : \R^n \setminus \tilde \Omega \to \R$ be the function that equals $t$ on $\tilde \Sigma_t$.  Note that $\tilde \varphi$ is continuous, since $R^{-1}$ is continuous (which holds because $V$, and hence $R$, is strictly increasing), $\tilde \varphi = 0$ on $\partial \tilde \Omega$, and $\tilde \varphi \to 1$ at infinity. We continue inequality 
(\ref{eqn_R_prime}), using the fact that $\omega_{n-1} R(t)^{n-1} = \int_{\tilde \Sigma_t} d\tilde A_t$ and the observation that $|\nabla \tilde \varphi| = \frac{1}{R'(t)}$ on $\tilde \Sigma_t$ for almost all $t \in [0,1)$:
\begin{align*}
 (n-2)\omega_{n-1} \capac(\Omega) &\geq  
    \int_0^1  \int_{\tilde \Sigma_t}|\nabla \tilde \varphi| d\tilde A_t dt &&(\text{by (\ref{eqn_R_prime})})\\
  &= \int_{\R^n \sm \tilde \Omega} |\nabla \tilde \varphi|^2 dV &&(\text{co-area formula}).
\end{align*}
Thus, $\tilde \varphi$ is in the Sobolev space $W^{1,2}_{\text{loc}}(\R^n \setminus \tilde \Omega)$. Note also that $\tilde \varphi$ is smooth near $\partial \tilde \Omega$ and is smooth outside a compact set. (These facts follow from $|\nabla \varphi| \neq 0$ near $\partial \Omega$, as $\partial_\nu(\varphi)>0$ on $\partial \Omega$ by the maximum principle, as well as $|\nabla \varphi| \neq 0$ near infinity, as $\varphi$ admits an expansion $\varphi(x) = 1 + \frac{c}{|x|^{n-2}} + O(|x|^{1-n})$ into spherical harmonics near infinity.) If $\tilde \varphi$ is not smooth on $\R^n \setminus \tilde \Omega$, proceed as follows. Let $U \subset \R^n \setminus \tilde \Omega$ be a smooth open set whose closure is compact and disjoint from $\partial \tilde \Omega$, where $U$ contains all points where $\tilde \varphi$ is not smooth. Using the density of smooth functions in $W^{1,2}(U)$, given any $\epsilon>0$, there exists a smooth function $\dbtilde{\varphi}$ that agrees with $\tilde \varphi$ near $\partial \tilde \Omega$ and outside a compact set, such that
$$ \int_{\R^n \sm \tilde \Omega} |\nabla \tilde \varphi|^2dV \geq  \int_{\R^n \sm \tilde \Omega} |\nabla \dbtilde{\varphi}|^2dV - \epsilon.$$
But $\dbtilde{\varphi}$ is a valid competitor in the definition of the capacity of $\tilde \Omega$, so that
$$\int_{\R^n \sm \tilde \Omega} |\nabla \dbtilde{\varphi}|^2dV \geq  (n-2)\omega_{n-1} \capac(\tilde \Omega).$$
Combining, we have
$$\capac(\Omega) \geq \capac(\tilde \Omega) - \frac{\epsilon}{(n-2)\omega_{n-1}}.$$
It is straightforward to check that equality holds in (\ref{ineq_cap_vol}) for round balls.  
Thus, since $\Omega$ has the same volume as $\tilde \Omega$, and $\epsilon$ can be made arbitrarily small, the proof is complete.
\end{proof}

\end{document}